\documentclass{article}

\usepackage{arxiv}

\usepackage{enumitem}
\usepackage[english]{babel}
\usepackage{mdframed}
\usepackage[utf8]{inputenc} 
\usepackage[T1]{fontenc}    
\usepackage{hyperref}       
\usepackage{url}            
\usepackage{booktabs}       
\usepackage{amsfonts}       
\usepackage{nicefrac}       
\usepackage{microtype}      
\usepackage{mathtools}
\usepackage{amsthm}
\usepackage{tikz}\usetikzlibrary{automata, positioning, arrows}
\usepackage{xcolor}
\usepackage{tikz}
\usepackage{pgfplots}

\DeclarePairedDelimiter\ceil{\lceil}{\rceil}

\newtheorem{theorem}{Theorem}[section]
\newtheorem{corollary}{Corollary}[theorem]
\newtheorem{lemma}[theorem]{Lemma}

\title{On the Ramsey Numbers of Odd-Linked Double Stars}

\author{
  Chaitanya D. Karamchedu \\
  Harvey Mudd College\\
  Claremont, CA, 91711 \\
  \texttt{ckaramchedu@hmc.edu} \\
  \And
  Maria M. Klawe \\
  Harvey Mudd College\\
  Claremont, CA, 91711 \\
 	\texttt{klawe@hmc.edu} \\
}

\begin{document}
\maketitle

\begin{abstract}
The linked double star $S_c(n,m)$, where $n \geq m \geq 0$, is the graph consisting of the union of two stars $K_{1,n}$ and $K_{1,m}$ with a path on $c$ vertices joining the centers. Its ramsey number $r(S_c(n,m))$ is the smallest integer $r$ such that every $2$-coloring of the edges of a $K_r$ admits a monochromatic $S_c(n,m)$. In this paper, we study the ramsey numbers of linked double stars when $c$ is odd. In particular, we establish bounds on the value of $r(S_c(n,m))$ and determine the exact value of $r(S_c(n,m))$ if $n \geq c$, or if $n \leq \lfloor \frac{c}{2} \rfloor - 2$ and $m = 2$.

\end{abstract}

\section{Introduction} \label{Introduction}
For any graph $G$, the \textit{ramsey number} $r(G)$ is the smallest positive integer $r$ such that any $2$-coloring of the edges of the complete graph on $r$ vertices ($K_r$) contains a monochromatic subgraph isomorphic to $G$. Famously, the exact determination of ramsey numbers is an exceedingly difficult computational problem that becomes intractable beyond small cases. As such, the ramsey numbers of various classes of graphs must be determined analytically, and some (such as the ramsey numbers of complete graphs, $r(K_n)$) have proved to be notoriously difficult to determine; as such, there are many open or partially open problems. However, the ramsey numbers of many classes of simple graphs (including stars, paths, and cycles) have been determined exactly: see \cite{Radziszowski} for a survey of known results of ramsey numbers as of January 2021.

In 1979, Grossman, Harary, and Klawe investigated the ramsey numbers of one such family of graphs known as double stars (\cite{GHK}). A \textit{double star} is defined as the union of two disjoint stars with an edge joining their centers. Grossman et al. solved the ramsey numbers of double stars in most cases, and more recently Norin, Sun, and Zhao have made some progress in the remaining open cases for double stars (\cite{Norin}). See Section \ref{Unsolved Problems} for more information. Motivated by the results of Grossman et al. and Norin et al., here we investigate the ramsey numbers of a generalization of double stars, which we refer to as \textit{linked double stars}.

\begin{figure}[ht]
\centering
\begin{center}
\begin{tikzpicture}[transform shape,scale = 0.5, every node/.style={scale=0.7}]
    \node[circle,fill,inner sep=0.3cm] (center) at (0,0) {};
\foreach \phi in {1,...,7}{
    \node[circle,fill,inner sep=0.3cm]  (v_\phi) at (360/8  * \phi:3cm) {};
         \draw[black] (v_\phi) -- (center);
      }
    \node[circle,fill,inner sep=0.3cm] (center2) at (3,0) {};
    \node[circle,fill,inner sep=0.3cm] (center3) at (7,0) {};
    \node[circle,fill,inner sep=0.3cm] (center5) at (14,0) {};
    \draw[black] (center) -- (center2);
    \foreach \phi in {1,...,5}{
    \node[shift = {(20,0)}, circle,fill,inner sep=0.3cm]  (v1_\phi) at (36 + 360/5  * \phi:3cm) {};
         \draw[black] (v1_\phi) -- (center5);
      }
    \draw[black] (center2) -- (v1_2);
   \end{tikzpicture} 
\end{center}
\caption{Example of a $S_5(7,4)$}
\end{figure}
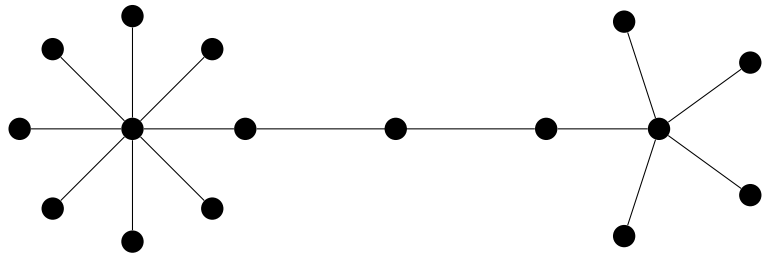

We define a linked double star as the union of two disjoint stars with a \textit{path} joining their centers. In particular, for $n \geq m \geq 0$, the \textit{linked double star} $S_c(n,m)$ is the graph on the points $\{ a_1, \ldots, a_c, v_1, \ldots, v_n, w_1, \ldots w_m  \}$ with the edges

$$\{ (a_1, v_i), (a_1, a_2), (a_2, a_3), \ldots, (a_{c-1}, a_c), (a_c, w_j) : 1 \leq i \leq n, 1 \leq j \leq m\}.$$ Note that above, $a_1$ and $a_c$ are the centers of the $n$-star and $m$-star respectively, and the path from $a_1$ to $a_c$ is the link between the two stars. For reference, Figure 1 shows an example of a $S_5(7,4)$.

In this paper, we will primarily be studying the ramsey numbers of \textit{odd-linked double stars}. Given the above definition of a linked double star, an odd-linked double star simply refers to a linked double star wherein the link has an odd number of vertices (i.e. $c = 2p + 1$ for some $p \in \mathbb{N}$).

We exactly determine the ramsey numbers of two groups of odd-linked double stars. In particular, if $c = 2p + 1$, we find that \begin{enumerate}
    \item $r(S_c(n,m)) = 2(n + m) + c - 2$ if $n \geq c$
    \item $r(S_c(n,2)) = n + 3p + 3$ if $n \leq p-2$. 
\end{enumerate}

In Section 2, we establish lower bounds on the ramsey numbers of odd-linked double stars, which hold irrespective of the relative size of the stars and the link. In Section 3, we establish upper bounds, which in the particular cases of the sizes of $n$ and $m$ described above, happen to be precisely equal to the lower bounds determined in Section 2 (thus determining the ramsey numbers exactly). Finally, in Section 4, we list open problems suggested by this work.

\section{Lower bounds} \label{Lower bounds}

In this section, we will establish lower bounds on the ramsey numbers of all odd-linked double stars. From this point forth, we will assume (and use interchangeably) $c = 2p + 1$ for some $p \in \mathbb{N}$, where again $c$ is the number of vertices on the link. The general statement of the lower bounds are presented in Theorem 2.1, the proof of which will be provided by applying the subsequent lemmas.

\begin{theorem}\label{lower bounds}
The ramsey numbers of the odd-linked double stars satisfy

$$r(S_c(n,m)) \geq \begin{cases}
2(n + m + p) - 1 &\mbox{for } n + m \geq p + 2 \\
n + m + 3p + 1 &\mbox{for } n + m \leq p + 2
\end{cases}.$$
\end{theorem}

\begin{lemma}\label{lower bound 1}
$r(S_c(n,m)) \geq 2(n + m + p) - 1$.
\end{lemma}

\begin{proof}
Consider a $2$-coloring of $K_{2(n + m + p) - 2}$, where the red subgraph consists of $K_{n + m + p - 1} \cup K_{n + m + p - 1}$, so that the blue subgraph is the complete bipartite graph $K(n + m + p - 1, n + m + p - 1)$. It is easy to see that there is no red $S_c(n,m)$, since $S_c(n,m)$ is connected and has $n + m + c$ vertices. But there also cannot be a blue $S_c(n,m)$, since when a $S_c(n,m)$ is presented as a subgraph of a bipartite graph there must be $n + m + p$ vertices that belong to the same partition.
\end{proof}

\begin{lemma}\label{lower bound 2}
$r(S_c(n,m)) \geq n + m + 3p + 1$.
\end{lemma}

\begin{proof}
Consider a $2$-coloring of $K_{n +m +3p}$, where the red subgraph consists of $K_{p} \cup K_{n + m + 2p}$, so that the blue subgraph is the complete bipartite graph $K(p, n + m + 2p)$. It is easy to see that there is no red $S_c(n,m)$ since $S_c(n,m)$ is connected and has $n + m + c$ vertices. But there also cannot be a blue $S_c(n,m)$, since when a $S_c(n,m)$ is presented as subgraph of a bipartite graph there must be $n + m + p$ vertices that belong to one partition, and at least $p+1$ vertices belonging to the other partition.
\end{proof}

Combining the results of Lemmas \ref{lower bound 1} and \ref{lower bound 2}, the proof of Theorem \ref{lower bounds} is completed.

\section{Upper bounds} \label{Upper bounds}
In this section, we will establish upper bounds on the ramsey numbers of odd-linked double stars.

We begin, as before, by stating the main theorem(s), and then prove the result(s) by a series of lemmas which together imply the desired result(s).

\begin{theorem}\label{upper bound 1}
If $n \geq 2p + 1$, the ramsey numbers of odd-linked double stars satisfy $$r(S_c(n,m)) \leq 2(n + m + p) - 1.$$
\end{theorem}

\begin{theorem}\label{upper bound 2}
If $n \leq p-2$, the ramsey numbers of odd-linked double stars satisfy $$r(S_c(n,2)) \leq n + 3p + 3.$$
\end{theorem}

\subsection*{Proof of Theorem \ref{upper bound 1}}

We will first introduce some notation which will be used throughout this section, the approach taken here is largely motivated by the strategy used by Grossman et al. in \cite{GHK}. Consider some fixed $2$-coloring of a complete graph $K$, and let $V$ denote the vertex-set of $K$. For a particular vertex $v \in V$, let $R(v)$ and $B(v)$ denote the sets of vertices joined to $v$ by a red and blue edge respectively. In this section, we may often use red-$d(v)$ to denote $|R(v)|$ and similarly for blue-$d(v)$. For a particular subset of points, $W \subseteq V$, define red-$d_W(v)$ as $$\text{red-}d_W(v) = |R(v) \cap W|$$ and similarly for blue-$d_W(v)$. We will also consider a particular vertex $u \in V$, where $u$ is chosen to have maximum monochromatic degree. Thus we may assume without loss of generality that $\text{red-}d(u) \geq \text{red-}d(v)$ and $\text{red-}d(u) \geq \text{blue-}d(v)$ for all $v \in V$. For convenience, we will define $A = R(u)$ and $B = B(u)$, and write $$|A| = n + m + p - 1 + k$$ for some integer $k \geq 0$. Finally, unless explicitly stated otherwise, we shall assume that $p \geq 1$ and $m \geq 1$.

We will now provide a series of lemmas that will be used repeatedly in this section.

\begin{lemma}\label{Yu_Li}(\textbf{Due to Yu and Li}, see \cite{Yu and Li})
A broom graph, $B(n,c)$, is equivalent to $S_c(n,0)$. The ramsey numbers of brooms are

$$r(B(n,c)) = 
\begin{dcases}
	n + c + \ceil*{\frac{c}{2}} - 1 & \text{if }c \geq 2n - 1 \\
	2(n + c) - 2\ceil*{\frac{c}{2}} - 1 & \text{if } 4\leq c \leq 2n - 2
\end{dcases}.$$ Note that if $n = 1$, or $1 \leq c \leq 2$, then the ramsey numbers of the corresponding brooms can be solved exactly by the ramsey numbers of paths and stars. See \cite{Yu and Li} for more information. If $c = 3$, then since $B(n,3) = S_{2}(n,1)$, the ramsey numbers of the corresponding brooms are solved by the results of Grossman et al. See \cite{GHK} for more information.

\end{lemma}

\begin{lemma}\label{k<n-p}
If $0 \leq k \leq n-p-1$, and $n \geq p+1$, then every such $2$-coloring of a $K_{r}$ admits a monochromatic $S_c(n,m)$, for $r \geq 2(n + m + p) -1$.
\end{lemma}

\begin{proof}
Clearly it is enough to prove the lemma for $r = 2(n + m + p) -1$. We will prove this lemma by induction.

\begin{itemize}
\item \underline{Base Case:}

Let $m = 1$. Then we wish to show that every $2$-coloring of a $K_{2n + 2p + 1}$ satisfying $k \leq n-p-1$ contains a monochromatic $S_{c}(n,1)$. Observe that $S_c(n,1)$ is equivalent to the broom graph $B(n, c+1)$. By Lemma \ref{Yu_Li}, we know that since $n \geq p + 1$, every $2$-coloring of a $K_{2n + 2p + 1}$ admits a monochromatic $B(n, c+1)$. Thus the base case is completed.

\item \underline{Inductive Step:}

Now, suppose the proposition holds for some fixed $m - 1$, and any $n \geq p + 1$. I.e., suppose that any $2$-coloring of a $K_{2n + 2(m -1) + 2p - 1}$, with $k \leq n-p-1$, must admit a monochromatic $S_c(n,m-1)$. Then we wish to show that any $2$-coloring of a $K_{2(n + m + p) - 1}$ must admit a monochromatic $S_c(n,m)$, if $n \geq p + 1$ and $k \leq n-p-1$. 

Consider some fixed $2$-coloring of a $K_{2(n + m + p)-1}$. From the induction hypothesis, since
\begin{align*}
k &\leq n-p-1 \\
&< (n + 1) - p - 1
\end{align*}
the $2$-coloring must admit a monochromatic $S_c(n+1, m-1)$, which we may assume is red. Let $Q$ be the red $S_c(n+1,m-1)$ and let $x$ be the vertex in $Q$ at the end of the path of length $2p + 1$ that begins at the $(n+1)$-star. Now, suppose to the contrary that this $2$-coloring avoided a monochromatic $S_c(n,m)$. In order to avoid a red $S_c(n,m)$, it must be the case that $x$ be blue-adjacent to the $n+1$ spoke vertices of the $n+1$ star in $Q$, as well as every vertex in $V \backslash Q$. Thus, 
\begin{align*}
\text{blue-}d(x) &\geq n + 1 + |V \backslash Q| \\
&= 2n + m - 1.
\end{align*}
 But since (by the definition of $k$) the maximum monochromatic degree of any vertex is $n + m + p - 1 + k$, it must be the case that $$n + m + p - 1 + k \geq 2n + m - 1.$$ Solving for $k$,
\begin{align*}
k &\geq n - p
\end{align*}
a contradiction.

\end{itemize}

\end{proof}

\begin{corollary}\label{k<p+1}
If $0 \leq k \leq p$, and $n \geq 2p + 1$, then every such $2$-coloring of a $K_r$ admits a monochromatic $S_c(n,m)$, for $r \geq 2(n + m + p) -1$.
\end{corollary}

\begin{lemma}\label{Jackson - 1}(\textbf{Due to Jackson}, see \cite{Jackson})
Let $G(a, b , k)$ is a simple bipartite graph with bipartition $(A, B)$ where $|A| = a \geq 2$, $|B| = b \geq k$, and each vertex of $A$ has degree at least $k$. If a graph $G(a,b,k)$ satisfies $a \leq k$ and $$b \leq 2k - 2$$ then it contains a cycle of length $2a$.
\end{lemma}

\begin{lemma}\label{No Red Cycle}
Consider some fixed $2$-coloring of a $K_{2(n + m +p) - 1}$. If $n \geq p + 1$, $k \geq p + 1$, and if there is a red cycle $Q$ in $A$ of length $2p$, then this $2$-coloring admits a monochromatic $S_c(n,m)$.
\end{lemma}

\begin{proof}
If $k \geq p + 1$, then $|A| \geq n + m + 2p$. Suppose there was a red cycle $Q \in A$ of length $2p$. Let $X = (A \cup B)\backslash Q$. Observe $|X| = 2n + 2m - 2$. If any $q \in Q$ satisfies $\text{red-}d_X(q) \geq m$, then there exists a red $S_c(n,m)$ formed by following $Q$ back for $2p$ vertices, starting from $q$, and connecting the last vertex back to $u$. This defines a red $S_c(n,m)$ with its $m$-star at $q$ and its $n$-star at $u$. Therefore all $q \in Q$ satisfy 
\begin{align*}
\text{blue-}d_X(q) &\geq 2n + m - 1\\
&\geq n + m + p.
\end{align*}
Now if there is a blue path of length $2p + 1$ between $Q$ and $X$ with both endpoints in $Q$, then there exists a blue $S_c(n,m)$ between $Q$ and $X$. Since $2n + m - 1 \leq |X| \leq 2(2n + m - 1) - 2$, we can conclude by Lemma \ref{Jackson - 1} that there is a blue cycle of length $4p$ between $Q$ and $X$, which contains in it a blue path of length $2p + 1$ with both endpoints in $Q$.

\end{proof}

\begin{lemma}\label{No Red Cycle - 2}

Consider some fixed $2$-coloring of a $K_{2(n + m + p) - 1}$. If $n \geq p + 3$, $k \geq p + 1$, and there is a red cycle $Q$ of length $2p + 2$ in $A \cup B$ with alternating vertices in $A$, then this $2$-coloring admits a monochromatic $S_c(n,m)$.

\end{lemma}

\begin{proof}
If $k \geq p + 1$, then $|A| \geq n + m + 2p$. Suppose there was a red cycle $Q$ of length $2p + 2$ in $A \cup B$, with alternating vertices in $A$. Let $X = (A \cup B) \backslash Q$,  let $Y \subset A$ be the set of alternating vertices on $Q$ in $A$, and let $Z = Q \backslash Y$. Observe that $|X| = 2n + 2m - 4$, and $|Y| = |Z| = p + 1$. If any $z \in Z$ satisfies $\text{red-}d_X(z) \geq m$, then there exists a red $S_c(n,m)$ formed by following $Q$ back for $2p$ vertices. This will end at some $y \in Y$, and adding the edge $(u,y)$ will define the link of a red $S_c(n,m)$ with the $n$-star at $u$ and the $m$-star at $z$. Therefore we may assume that all $z \in Z$ satisfy 
\begin{align*}
    \text{blue-}d_X(z) &\geq 2n + m - 3\\
    &\geq n + m + p.
\end{align*}
Now if there is a blue path of length $2p + 1$ between $Z$ and $X$, with both endpoints in $Z$, then there exists a blue $S_c(n,m)$ between $Z$ and $X$. Since $2n + m - 3 \leq |X| \leq 2(2n + m - 3) - 2$, we can conclude by Lemma \ref{Jackson - 1} that there is a blue cycle of length $2p + 2$ between $Z$ and $X$, which contains in it a blue path of length $2p + 1$ with both endpoints in $Z$.
\end{proof}

\begin{lemma}\label{Ramsey numbers of cycles-1}(\textbf{Due to Faudree, Shelp, and Rosta}, see \cite{FR_cycles} and \cite{Rosta})
Let $C_n$ and $C_m$ denote cycles of length $n$ and $m$ respectively. Then the ramsey numbers of cycles satisfy

$$R(C_m, C_n) = \begin{cases}
2n - 1 &\mbox{for } 3 \leq m \leq n, m \mbox{ odd }, (m,n) \neq (3,3) \\
n - 1 + \frac{m}{2} &\mbox{for } 4 \leq m \leq n, m \text{ and } n \mbox{ even }, (m,n) \neq (4,4) \\
\text{max}(2m - 1, n - 1 + \frac{m}{2}) &\mbox{for } 4 \leq m < n, m \text{ even, and } n \mbox{ odd}
\end{cases}.$$
\end{lemma}

\begin{lemma}\label{degree-path}
Suppose $G$ is a graph in which every vertex has degree at least $s$. Let $z$ be any vertex in $G$. Then $G$ has a path of length $s+1$ with $z$ as an endpoint.
\end{lemma}

\begin{proof}
For $i = 1, \dots, s + 1$ we construct a path where $z_1 = z$, and $z_i$ is any vertex adjacent to $z_{i - 1}$. Since every vertex has degree at least $s$, this yields a path of length $s+1$.
\end{proof}

\begin{lemma}\label{No Blue Star - helper}
Consider some fixed $2$-coloring of a $K_{2(n + m + p) - 1}$. Suppose $n \geq p + 3$, $m \geq 2$, and $k \geq p + 1$. Assume there exists a blue $m + 2p$ star in $A$, centered at some $v \in A$, with the set of spoke vertices being $A'$. If there exists a blue cycle $Q$ of length $2p$ in $A$, such that $$|A' \cap Q| \geq p + 1,$$ then this $2$-coloring must admit a monochromatic $S_c(n,m)$. 
\end{lemma}

\begin{proof}
Let $D \subset A'$ be some set of $m$ vertices in $A'$ not on $Q$. Define $$X = V \backslash (Q \cup D \cup \{ v \}).$$ Observe that $|X| = 2n + m - 2$. Furthermore, let $Q' = A' \cap Q$, and let $Q'' \subset Q$ be the set of vertices that are adjacent to a vertex in $Q'$; note that since $|Q'| \geq p + 1$, it must also be the case that $|Q''| \geq p + 1$. Observe that if any $q \in Q''$ satisfies $\text{blue-}d_X(q) \geq n$ then there exists a blue $S_c(n,m)$, formed by following $q$ back for $2p$ vertices on $Q$, and connecting the last vertex to $v$. This defines a blue $S_c(n,m)$ with $n$-star at $q$ and $m$-star at $v$ (using $D$). Therefore, we may assume that all $q \in Q''$ satisfy $\text{blue-}d_X(q) \leq n - 1$, and so
\begin{align*}
\text{red-}d_X(q) &\geq 2n + m - 2 - (n - 1) \\
&= n + m - 1.
\end{align*}
Since $m \geq 2$, $$n + m - 1 \leq |X| \leq 2(n + m - 1) - 2$$ and so by Lemma \ref{Jackson - 1} there must be a red cycle $S$ of length $2p + 2$ between $Q''$ and $X$. Now one of two things must be true: either $u \notin S$ or $u \in S$.

Suppose $u \notin S$. Then $S$ is a red cycle of length $2p +2$ between $Q''$ and $A \cup B$. Since $Q'' \subset A$, we can conclude by Lemma \ref{No Red Cycle - 2} that this $2$-coloring admits a monochromatic $S_c(n,m)$.

Alternatively, suppose $u \in S$. Define $$Y = V \backslash (Q \cup D \cup S \cup \{ v \}).$$ Observe that $|Y| = 2n + m - p - 3$. Following $S$ for $2p + 1$ vertices from $u$ will end at some $x \in (S \cap X)$. Now if $\text{red-}d_Y(x) \geq m$ then there exists a red $S_c(n,m)$ with the $n$-star at $u$ and $m$-star at $x$. Therefore we may assume $\text{red-}d_Y(x) \leq m - 1$, and so 
\begin{align*}
\text{blue-}d_Y(x) &\geq (2n + m - p - 3) - (m - 1) \\
&= 2n - p - 2 \\
&\geq n + 1.  
\end{align*}
Let $T = B(x) \cap Y$. Finally, define $Q'''$ to be the set of vertices $q \in Q$ such that following $Q$ for $2p - 2$ vertices from $q$ ends at a vertex in $Q'$. I.e., if $q \in Q'''$, then there exists a vertex $w \in Q'$ such that $q$ and $w$ are the endpoints of a blue path of length $2p - 2$ in $Q$. Observe that since $|Q'| \geq p + 1$, it must also be the case that $|Q'''| \geq p + 1$. Now if any $t \in T$ is connected to any $q \in Q'''$ by a blue edge, then connecting $t$ to $q$, following $Q$ back for $2p - 2$ vertices from $q$, and connecting the last vertex to $v$, will form a blue $S_c(n,m)$ with the $n$-star at $x$ and $m$-star at $v$. Observe that $x$ has sufficient blue degree to form the $n$-star without interference since $Y$ and $Q$ are disjoint. Therefore we may assume there is a red complete bipartite graph between $Q'''$ and $T$. This complete bipartite graph contains a red cycle of length $2p + 2$, and since $Q''' \subset A$ and $T \subset (A \cup B)$, we may conclude by Lemma \ref{No Red Cycle - 2} that there is a monochromatic $S_c(n,m)$.
\end{proof}

\begin{lemma}\label{No Blue Star}
Consider some fixed $2$-coloring of a $K_{2(n + m + p) - 1}$. If $n \geq p + 3$, $m \geq 2$, and $k \geq p + 1$, then if there is a blue $m + 2p$ star in $A$, then this $2$-coloring admits a monochromatic $S_c(n,m)$.
\end{lemma}

\begin{proof}
Suppose that there was such a blue $m + 2p$ star in $A$ centered at some vertex $v \in A$. Let the set of $m + 2p$ spoke vertices be $A' \subset A$.

First suppose that $m \geq p - 1$. Then $|A'| \geq 3p - 1$, and so by Lemma \ref{Ramsey numbers of cycles-1} we know there must either be a red cycle of length $2p$ in $A'$, or a blue cycle of length $2p$ in $A'$. If there is a red cycle of length $2p$ in $A'$, then by Lemma \ref{No Red Cycle} we could conclude that there is a monochromatic $S_c(n,m)$. Therefore we may assume there is no such red cycle, and so therefore there is a blue cycle $Q \subset A'$ of length $2p$. Now $A'$ is a blue $m + 2p$ star in $A$ such that there exists a blue cycle $Q$ in $A$ satisfying $|A' \cap Q| \geq p + 1$. Thus by Lemma \ref{No Blue Star - helper}, this $2$-coloring must admit a monochromatic $S_c(n,m)$.

Now assume that $m \leq p - 2$. Choose some set of $p - m - 1$ vertices from $A \backslash (A' \cup\{ v \} )$. Call this set $D$. Set $A'' = D \cup A',$ and observe that $|A''| = 3p -1$, and so again by Lemma \ref{Ramsey numbers of cycles-1}, there must either be a red cycle of length $2p$ in $A''$, or a blue cycle of length $2p$ in $A''$. If red, then by Lemma \ref{No Red Cycle} we could conclude that there is a monochromatic $S_c(n,m)$. Therefore we may assume there is a blue cycle of length $2p$ in $A''$. Call this blue cycle $Q$, and observe that $$|Q \cap D| \leq p- m -1$$ and so $|Q \cap A'| \geq p + m + 1$. Now $A'$ is a blue $m + 2p$ star in $A$ such that there exists a blue cycle $Q$ in $A$ satisfying $|A' \cap Q| \geq p + 1$. Thus by Lemma \ref{No Blue Star - helper}, this $2$-coloring must admit a monochromatic $S_c(n,m)$.

\end{proof}

\begin{lemma}\label{p < k < 2p + 1}
If $p+1 \leq k \leq 2p$, $n \geq 2p + 1$, and $m \geq 2$, then every such $2$-coloring of a $K_r$ admits a monochromatic $S_c(n,m)$, for $r \geq 2(n + m + p) -1$.
\end{lemma}

\begin{proof}
Clearly it is enough to prove the lemma for $r = 2(n + m + p) - 1$. If $p+1 \leq k \leq 2p$, then $n + m + 2p \leq |A| \leq n + m + 3p - 1$. Since $u$ has maximal monochromatic degree, $$\text{red-}d(v) \geq n + m - p -1$$ for all $v \in V$. By Lemma \ref{No Blue Star} we may assume there is no blue $m + 2p$ star in $A$. Therefore, for all $a \in A$, we have $\text{red-}d_A(a) \geq n$. 

By Lemma \ref{Yu_Li}, there must be a monochromatic broom $B(m, 2p + 1)$ in $A$. If red, then connecting the second-to-last vertex of the path of the broom to $u$ completes a red $S_c(n,m)$. Thus we may there exists a blue $B(m, 2p + 1)$ in $A$. Let the last vertex of the path of this broom be $v$, let the set of $m + 2p + 1$ vertices that define this broom be $Q$, and let $X = A \cup B \backslash Q$. Note that $|X| = 2n + m - 3$. Now clearly if $\text{blue-}d_X(v) \geq n$, then $v$ completes a blue $S_c(n,m)$, with $Q$ forming the remainder of the linked double star. Therefore we may assume $\text{blue-}d_X(v) \leq n - 1$, and so
\begin{align*}
\text{red-}d_X(v) &\geq 2n + m - 3 - (n -1) \\
&= n + m - 2 \\
&\geq m + 2p - 1.
\end{align*}
Finally, note that since $v \in A$, and $ \forall a \in A$
\begin{align*}
\text{red-}d_A(a) &\geq n \\
&\geq 2p + 1,
\end{align*}
by Lemma \ref{degree-path} there must be a red path of length $2p$ in $A$, with one endpoint somewhere in $A$ and the other endpoint at $v$. Since $\text{red-}d_X(v) \geq m + 2p - 1$, this result implies that there exists a red $B(m, 2p)$ in $A \cup B$, with the endpoint of the tail somewhere in $A$. Connecting the endpoint of the tail to $u$ will complete a red $S_c(n,m)$.
\end{proof}

\begin{lemma}\label{Ramsey numbers of paths versus stars}(\textbf{Due to Parsons}, see \cite{Parsons})
The ramsey numbers of paths vs. stars satisfy

$$r(P_l, K_{1,n}) \leq n + l - 1.$$
\end{lemma}

\begin{lemma}\label{k > 2p}
If $k \geq 2p + 1$, and $n \geq p+3$, then every such $2$-coloring of a $K_r$ admits a monochromatic $S_c(n,m)$, for $r \geq 2(n + m + p) - 1$.
\end{lemma}

\begin{proof}
Clearly, it suffices to prove the result for $r = 2(n + m + p) - 1$. If $k \geq 2p + 1$, then $|A| \geq n + m + 3p$. By Lemma \ref{No Blue Star}, we may assume that for all $a \in A$, $$\text{blue-}d_A(a) \leq m + 2p - 1$$ and so 
\begin{align*}
\text{red-}d_A(a) &\geq n - p + k - 1\\
&\geq n+p+1.
\end{align*}
However, if there is no blue $m + 2p$ star in $A$, then by Lemma \ref{Ramsey numbers of paths versus stars} there must be a red path $P$ of length $2p$ in $A$. Let the vertices of $P$ be $$P = \{ v_1, v_2, \ldots, v_{2p - 1}, v_{2p} \}$$ where $P$ has endpoints $v_1$ and $v_{2p}$. Since $v_1, v_{2p} \in A$,
\begin{align*}
\text{red-}d_A(v_1) &\geq n - p + k - 1 \\
&\geq m + p
\end{align*}
(and similarly for $v_{2p}$). Now if either $\text{red-}d_P(v_1) \leq p$ or $\text{red-}d_P(v_{2p}) \leq p$, then one of $v_1$ or $v_{2p}$ forms the center of a red $m$-star, and connecting the other endpoint of $P$ to $u$ will complete a red $S_c(n,m)$. Therefore we may assume that both $\text{red-}d_P(v_1) \geq p+1$ and $\text{red-}d_P(v_{2p}) \geq p+1$. Let $X = R(v_1) \cap P$, and $Y = R(v_{2p}) \cap P$. Now we claim that there must exist some vertex $v_i \in X$, such that $v_{i-1} \in Y$. Suppose to the contrary that there was no such $v_i$. Observe that every $v_i \in X$ corresponds uniquely to a $v_{i - 1}$ that (by supposition) must be connected to $v_{2p}$ by a blue edge. However, this would imply that $\text{blue-}d_P(v_{2p}) \geq p$ and so $\text{red-}d_P(v_{2p}) \leq p-1$, a contradiction. Therefore, we may assume there exists some $v_i \in X$ such that $v_{i-1} \in Y$. But now these vertices can be used to define a red cycle $Q$, as follows:
$$Q = \{ v_1, v_2, \ldots, v_{i-1}, v_{2p}, v_{2p - 1}, \ldots, v_{i + 1}, v_{i}, v_1 \}.$$Observe that $Q \subset A$ is a red cycle of length $2p$, and so by Lemma \ref{No Red Cycle}, this $2$-coloring must admit a monochromatic $S_c(n,m)$.
\end{proof}

Combining the results of Lemma \ref{Yu_Li}, Corollary \ref{k<p+1}, Lemma \ref{p < k < 2p + 1}, and Lemma \ref{k > 2p}, the proof of Theorem \ref{upper bound 1} is complete.

\subsection*{Proof of Theorem \ref{upper bound 2}}
As before, we will begin by introducing notation that will be used throughout this section. Consider some fixed $2$-coloring of a complete graph $K$, and let $V$ denote the vertex-set of $K$. In this section, we will assume that $|V| \geq 3p + 1$. Therefore, by lemma \ref{Ramsey numbers of cycles-1}, there must exist a monochromatic cycle of length $2p+2$ in this $2$-coloring.

For every monochromatic cycle in the $2$-coloring of $K$, we will also identify an associated ``primed" vertex. In particular, if $C$ is a blue cycle of length $2p+2$ in $K$, then its primed vertex $u \in V \backslash C$, is the vertex such that $\text{blue-}d_C(u) \geq \text{blue-}d_C(v)$ for all $v \in V \backslash C$. Similarly, if $C$ is a red cycle of length $2p+2$ then its primed vertex $u \in V \backslash C$ is the vertex such that $\text{red-}d_C(u) \geq \text{red-}d_C(v)$. For a particular monochromatic cycle $C$ of length $2p + 2$, its primed vertex will be denoted $u_C$. Now, using $u_C$, we will define three important sets of vertices in $C$. First, if $C$ is a blue cycle we will define $S_C = B(u_C) \cap C$, and similarly if $C$ is a red cycle, $S_C = R(u_C) \cap C$. Second, we will define the set $T_C^i$ as $$T_C^i = \{ c \in C : \exists s \in S_C \text{ such that } c \text{ and } s \text{ are endpoints of a path of length } i \text{ on } C \}.$$ Finally, define $Z_C = V \backslash C$. Note that $u_C \in Z_C$.

Our proofs focus on a particular cycle $Q$, with length $2p + 2$, whose primed vertex is maximal. Specifically, if $H$ is the set of all monochromatic cycles of length $2p+2$ in the $2$-coloring of $K$, then $Q \in H$ is the cycle such that $|S_Q| \geq |S_C|$ for all $C \in H$. We will assume, without loss of generality, that $Q$ is blue. For the sake of concision, we shall refer to $u_Q$ simply as $u$, $S_Q$ as $S$, $T_Q^i$ as $T^i$, and $Z_Q$ as $Z$.

We will now provide a series of lemmas that will be used repeatedly in this section. Unless explicitly stated otherwise, we shall assume that $n \leq p-2$.

\begin{lemma} \label{S<p}
If $|S| \leq p-1$, and $n \leq p$, then every such $2$-coloring of a $K_r$ admits a monochromatic $S_c(n,2)$, for $r \geq n + 3p + 3$.
\end{lemma}

\begin{proof}
Clearly it is enough to prove the lemma for $r = n + 3p + 3$. If $|S| \leq p-1$, then by the definition of $S$, it implies that $\forall z \in Z$, $\text{blue-}d_Q(z) \leq p-1$ and so $\text{red-}d_Q(z) \geq p+3$. Now observe that $$p+3 \leq |Q| \leq 2(p+3) - 2.$$ Therefore, by Lemma \ref{Jackson - 1}, any subset of $p+1$ vertices in $Z$ will define a red cycle of length $2p+2$ between that subset and $Q$.

Now suppose there is some vertex $z \in Z$, satisfying $\text{red-}d_Z(z) \geq p$. Let $Z' \subset Z$ be a set of $p+1$ vertices in $Z \backslash \{ z \}$, such that $\text{red-}d_{Z'}(z) \geq p$. By the above reasoning, there must exist a red cycle of length $2p+2$ between $Z'$ and $Q$, call this red cycle $Q'$. Since $z \notin Q'$, and $\text{red-}d_{Q'}(z) \geq p$, we know that $|S_{Q'}| \geq p$, contradicting the maximality of $Q$. Therefore, we may assume that $\forall z \in Z$, $\text{red-}d_Z(z) \leq p-1$, and so $\text{blue-}d_Z(z) \geq n+1$.

Similarly, suppose there is some $z \in Z$ such that $\text{red-}d_Q(z) = 2p + 2$. Let $Z' \subset Z$ be some set of $p+1$ vertices in $Z \backslash \{ z \}$. There must exist a red cycle of length $2p + 2$ between $Z'$ and $Q$, call this red cycle $Q'$. Since $z \notin Q'$, and $\text{red-}d_{Q'}(z) \geq p + 1$, we know that $|S_{Q'}| \geq p+1$, contradicting the maximality of $Q$. Therefore, we may assume that $\forall z \in Z$, $\text{blue-}d_Q(z) \geq 1$, and so every $z \in Z$ must be connected to some $q \in Q$ by a blue edge.

Consider some fixed vertex $z \in Z$. Since $\forall z \in Z$ we may assume that $\text{blue-}d_Z(z) \geq n+1$, $z$ is the center of (at least) a blue $n + 1$ star. Furthermore, let $q \in Q$ be a vertex such that the edge $(z, q)$ is blue, which must exist since $\forall z \in Z$, $\text{blue-}d_Q(z) \geq 1$. Connecting $z$ to $q$ and following $Q$ for $2p$ vertices from $q$ defines a blue path of length $2p+1$ that begins at $z$ and ends at some $q' \in Q$. Observe that this path has an $n+1$-star at $z$, so if $\text{blue-}d_{Z \backslash \{ z \}}(q') \geq 1$ then this path would define a blue $S_c(n,2)$, with $2$-star at $q'$ using one vertex in $Z \backslash \{ z \}$ and the vertex on $Q$ following $q'$. Therefore we may assume $\text{blue-}d(q')_{Z \backslash \{ z \}} = 0$, and so $\text{red-}d(q')_{Z \backslash \{ z \}} = n + p$. Now select some arbitrary set of $p + 1$ vertices from these $n + p$, call this set $A$. Since $\forall z \in Z$ satisfy $\text{red-}d_Q(z) \geq p + 3$, it follows that $\text{red-}d_{Q \backslash \{ q'  \}} \geq p + 2$. But now, since $$p + 2 \leq |Q \backslash \{ q' \}| \leq 2(p + 2) - 2$$ there must be a red cycle $Q'$ of length $2p + 2$ between $A$ and $Q \backslash \{ q' \}$. Since $q' \notin Q'$, and $\text{red-}d_{Q'}(q') \geq p + 1$, we know that $|S_{Q'}| \geq p+1$, contradicting the maximality of $Q$.
\end{proof}

\begin{lemma}\label{|T_C|}
Let $C$ be a monochromatic cycle of length $2p+2$. Then $\forall i$, $|T_C^i| \geq |S_C|$.
\end{lemma}

\begin{proof}
Every vertex in $S_C$ contributes (at least) one unique vertex to $T_C^i$. Simply fix a direction for following $C$, and follow $C$ for $i$ vertices starting from each vertex in $S_C$. In this manner, each vertex $s \in S_C$ will map onto a unique vertex that (by the definition of $T_C^i$) must belong to $T_C^i$. Therefore, $|T_C^i| \geq |S_C|$. 
\end{proof}

\begin{lemma}\label{4-neighbors}
Suppose $C$ is a cycle on $2p+2$ vertices. If $|S_C| = |T_C^{2p}|$, then $|S_C| \geq p+1$.
\end{lemma}

\begin{proof}
First, observe that $T_C^{2p} = T_C^{4}$, since as $C$ is a cycle, any two vertices that form the endpoints of a path of length $2p$ on $C$ also form a path of length $4$ on $C$. It is therefore equivalent to prove the result for $T_C^4$.

Consider some fixed $s \in S_C$. Observe that $s$ defines exactly two vertices that belong to $T_C^4$: one obtained by following $C$ for $4$ vertices starting from $s$ in the ``forward" direction, and the other by following $C$ in the ``backward" direction. Call these two vertices the forward and backward $4$-neighbors of $s$ respectively. Since in this manner every $s \in S$ defines two vertices in $T_C^4$, the only way that $|S_C| = |T_C^4|$ is if every forward $4$-neighbor of a vertex in $S$ is \emph{also} a backward $4$-neighbor of some other vertex in $S$.

Now observe that if there are any four consecutive vertices $x_1, x_2, x_3, x_4 \in C \backslash S_C$, then $|T_C^4| > |S_C|$, since $x_4$ is a forward $4$-neighbor of $x_1$, and $x_4 \notin S$. This implies that if there are three adjacent vertices $x_1, x_2, x_3 \in C \backslash S_C$ then $S_C$ must contain the three vertices before $x_1, x_2, x_3$ and the three vertices after $x_1, x_2, x_3$. Similarly if there are two adjacent vertices $x_1,x_2 \in C \backslash S_C$ then $S_C$ must contain the three vertices before and after $x_1, x_2$. But now $C \backslash S_C$ consists of a sequence of intervals of $C$ of length one, two, or three, and each interval is followed by an interval in $S$ of equal or greater length, so $|S_C| \geq |C \backslash S_C|$. Since $|S_C| + |C \backslash S_C| = |C| = 2p + 2$, this necessitates that $|S_C| \geq p + 1$.
\end{proof}

\begin{corollary}\label{4-neighbor-cor}
If $C$ is a cycle on $2p+2$ vertices with $|S_C| \leq p$, then $|T_C^{2p}| \geq |S_C| + 1$.
\end{corollary}

\begin{proof}
By Lemma \ref{|T_C|}, we know $|T_C^{2p}| \geq |S_C|$. However, as shown in Lemma \ref{4-neighbors}, equality can only occur if $|S_C| \geq p+1$, so if $|S_C| \leq p$ then $|T_C^{2p}| \geq |S_C| + 1$.
\end{proof}

\begin{lemma}\label{S=p}
If $|S| = p$, and $n \leq p$, then every such $2$-coloring of a $K_r$ admits a monochromatic $S_c(n,2)$, for $r \geq n + 3p + 3$.
\end{lemma}

\begin{proof}
Clearly it is enough to prove the lemma for $r = n + 3p + 3$. If $|S| = p$, then by the definition of $S$, it implies that $\forall z \in Z$, $\text{blue-}d_Q(z) \leq p$ and so $\text{red-}d_Q(z) \geq p + 2$. Now observe that $$p+2 \leq |Q| \leq 2(p+2) - 2.$$ Therefore, by Lemma \ref{Jackson - 1}, any subset of $p+1$ vertices in $Z$ will define a red cycle of length $2p + 2$ between that subset and $Q$.

Suppose that there was some $z \in Z$ satisfying $\text{red-}d_Z(z) \geq p$. Let $A \subset Z$ be some subset of $p + 1$ vertices in $Z \backslash \{ z \}$, including at least $p$ vertices from $R(z) \cap Z$. By Lemma \ref{Jackson - 1} there must be a red cycle $Q'$ of length $2p+2$ between $A$ and $Q$. Now if $z$ is connected to a vertex in $Q' \cap Q$ by a red edge, then $|S_{Q'}| \geq p+1$ (since $z \notin Q'$, and $\text{red-}d_A(z) \geq p$) contradicting the presumed maximality of $Q$. Therefore, we may assume that $z$ is connected to every vertex in $Q' \cap Q$ by a blue edge. But now, $\text{blue-}d_Q(z) \geq p+1$, and so $|S| \geq p+1$, contradicting the assumption that $|S| = p$.

Thus, $\forall z \in Z$ we may assume that $\text{red-}d_Z(z) \leq p-1$ and so $\text{blue-}d_Z(z) \geq n+1$. Let $Z' = Z \backslash \{ u \}$. Suppose there as some vertex $z \in Z'$ that was connected to some vertex in $t \in T^{2p}$ by a blue edge. Then following $z$ to $t$, and following $Q$ for a further $2p$ vertices from $Q$ will end at some $s \in S$. This path forms the link of a blue $S_c(n,2)$, where we use $u$ and the remaining vertex on $Q$ that is blue-adjacent to $s$ to complete the $2$-star at $s$, leaving at least $n$ vertices to form the blue $n$-star at $z$ (note that since $z \in Z'$, it follows $z \neq u$). Thus we may assume there is a red complete bipartite graph between $Z'$ and $T^{2p}$. Since (by Corollary \ref{4-neighbor-cor}) $|T^{2p}| \geq p + 1$, and 
\begin{align*}
    |Z'| &= n + p\\ 
    &\geq p + 2,
\end{align*}
this complete bipartite graph contains a red cycle of length $2p + 2$, which leaves one vertex unused in $Z'$ which has red degree into the cycle $\geq p+1$, again contradicting the assumption of the maximality of $Q$.
\end{proof}

\begin{lemma}\label{red degree T^{2p+1}}
Suppose $C$ is a blue cycle of length $2p + 2$. If any vertex $t \in T_C^{2p+1}$ has $\text{blue-}d_{Z_C}(t) \geq n + 1$, then this $2$-coloring admits a blue $S_c(n,2)$. The symmetric result holds if $C$ is instead a red cycle of length $2p+2$.
\end{lemma}

\begin{proof}
Let $C$ be a blue cycle of length $2p+2$, and suppose there existed a $t \in T_C^{2p+1}$ such that $\text{blue-}d_{Z_C}(t) \geq n + 1$. By the definition of $T_C^{2p+1}$, following $C$ for $2p+1$ vertices from $t$ will end at some $s \in S_C$. Observe now that there exists a $2$-star at $s$, using $u$, and the remaining vertex on $Q$ that immediately follows $s$. This leaves (at least) $n$ vertices unused in $B(t) \cap Z_C$ to form the $n$-star at $t$, so this path on $Q$ defines a blue $S_c(n,2)$.
\end{proof}

\begin{lemma}\label{min-degree 1}
Suppose $C$ is a blue cycle of length $2p+2$. Let $T = \{ c \in C : \text{red-}d_{Z_C}(c) \geq p+1 \}$. If $|S_C| \geq p+1$ and $|T| \geq p+2$, then if there exists a vertex $c \in C$ such that $\text{blue-}d_{Z_C}(c) \geq n+1$, then this $2$-coloring admits a monochromatic $S_c(n,2)$. The symmetric result also holds if $C$ is instead a red cycle of length $2p+2$.
\end{lemma}

\begin{proof}
Suppose there existed a vertex $c_1 \in C$, with $\text{blue-}d_{Z_C}(c) \geq n+1$. Let $c_3$ and $c_{2p+1}$ be the vertices found by following $C$ for $2p+1$ vertices starting at $c_1$. If either of $c_3$ or $c_{2p+1}$ has a blue edge into $Z_C$, then it would complete a blue $S_c(n,2)$, with the $n$-star at $c_1$ and the $2$-star at either $c_3$ or $c_{2p+1}$ (where we have used $c_2$ or $c_{2p+2}$ to complete the $2$-star at $c_3$ or $c_{2p+1}$, respectively). Therefore we may assume that $c_3$ and $c_{2p+1}$ are connected to every $z \in Z_C$ by a red edge.

Suppose that $c_{2p+1}$ was connected to some vertex $c' \in C \backslash \{ c_3 \}$ by a red edge. Let $T' = T \backslash \{ c_{2p+1}, c' \}$. Since $|T| \geq p+2$, we have $|T'| \geq p$. Since $T' \subseteq T$, we know that $\forall t \in T'$, $\text{red-}d_{Z_C}(t) \geq p+1$. Since $$p+1 \leq |Z_C| \leq 2(p+1) - 2,$$ by Lemma \ref{Jackson - 1}, this implies that there exists a red cycle of length $2p$ between $T'$ and $Z_C$; call this red cycle $C'$. If $c_3 \in T'$, then $C'$ contains a path of length $2p$ with one endpoint at $c_3$, and the other endpoint at some $z_1 \in Z_C$. If $c_3 \notin T'$, then such a red path can be formed by following $C'$ for $2p-1$ vertices starting at some $z_1 \in Z_C \cap C'$; this will end at some other $z_2 \in Z_C$, to which we can append $c_3$ to form the desired red path. In either case, there exists a red path of length $2p$ between $T'$ and $Z_C$ with one endpoint at $c'$ and the other endpoint at some $z_1 \in Z_C$. Now this red path defines a red $S_c(n,2)$, as we can use $n$ of the $n+1$ unused vertices in $Z_C$ to complete the $n$-star at $c_3$, and can use $c'$ and the remaining unused vertex in $Z_C$ to complete the $2$-star at $c_{2p+1}$. Therefore we may assume that $c_{2p+1}$ must be connected to every vertex in $C \backslash \{ c_3 \}$ by a blue edge. By symmetric reasoning, we may also assume that $c_3$ is connected to every vertex in $C \backslash \{ c_{2p+1} \}$ by a blue edge. However, this implies that every vertex in $C \backslash \{ c_1 \}$ can be made the endpoint of a blue path of length $2p+1$ starting at $c_1$. Suppose that we wished to make some $c_i \in C \backslash \{ c_1 \}$ such an endpoint of a blue path starting at $c_1$. Then using $c_3$ and $c_{2p+1}$, the path $$P = \{ c_1, c_2, c_3, \ldots, c_{i-1}, c_{2p+1}, c_{2p}, \ldots, c_{i+1}, c_i \},$$ gives the desired blue path of length $2p+1$ (note a similar path exists if instead $c_i = c_2$ or $c_i = c_{2p+2}$. Thus, if any $c \in C \backslash \{ c_1 \}$ has $\text{blue-}d_{Z_C}(c) \geq 2$, there exists a blue $S_c(n,2)$ (note there is one vertex left unused on $C$, $c_{2p + 2}$, that may be used to complete the $n$-star at $c_1$, if needed). Thus we may assume that $\forall c \in C \backslash \{ c_1 \}$, $\text{blue-}d_{Z_C}(c) \leq 1$ and so $\text{red-}d_{Z_C}(c) \geq n + p$.

Now since $n < p$ we have $$\frac{(n + p)(2p+1)}{n + p + 1} > n + p + 1,$$ and so there must be a vertex $z' \in Z_C$ such that $\text{red-}d_C(z') \geq n + p + 2$. Let $W \subseteq C$ be a subset of $R(z') \cap C$ of size $n + p + 2$. Now since $$\frac{(n + p - 1) (n + p + 2)}{n + p} > 2$$there must exist a vertex $z \in Z_C \backslash \{ z' \}$ such that $\text{red-}d_W(z) \geq 3$. Let $\{ w_1, w_2, w_3 \} \subseteq R(z) \cap W$. Choose any $p$ vertices of $W$ including the vertex $w_1$, but excluding $w_2$ and $w_3$; call this set of vertices $A$. Observe that since $A \subset C$, we know $\forall a \in A$, $\text{red-}d_{Z_C \backslash \{ z' \}}(a) \geq n + p - 1 \geq p + 1$. Furthermore, since $$p + 1 \leq n + p \leq 2(p + 1) - 2,$$ by Lemma \ref{Jackson - 1} there must be a red cycle $C'$ of length $2p$ between $W$ and $Z_C \backslash \{ z' \}$ (where above we have used the fact that $n \geq 2$ and $n \leq p - 1$). Since $C'$ contains $w_1$, avoids $w_2$ and $w_3$, and $z$ is red-adjacent to $\{ w_1, w_2, w_3 \}$, this defines a red $S_{2p}(2, 0)$, with one endpoint at $ z \in Z_C \backslash \{ z' \}$ and the other endpoint somewhere in $W$. But since $W \subseteq R(z') \cap C$, we may append $z'$ to the endpoint in $W$, and since $|W| = n + p + 2$, $z'$ has sufficient degree to complete the red $S_{c}(n,2)$ without interference.

\end{proof}

\begin{corollary} \label{min-degree 2}
Let $C$ be a blue cycle of length $2p + 2$. If $|S_C| \geq p+1$ and $|T_C^{2p+1}| \geq p+2$, then if there exists a vertex $c \in C$ such that $\text{blue-}d_{Z_C}(c) \geq n + 1$, then this $2$-coloring admits a monochromatic $S_c(n,2)$.
\end{corollary}

\begin{proof}
By Lemma \ref{red degree T^{2p+1}}, we may assume that every $t \in T^{2p+1}$ satisfies $\text{red-}d_{Z_C}(t) \geq p+1$. In particular, if $T = \{ c \in C : \text{ red-}d_{Z_C}(c) \geq p+1 \}$, then $|T| \geq |T_C^{2p+1}|$. Therefore if $|T_C^{2p+1}| \geq p+2$, then $|T| \geq p+2$, and so by the result of Lemma \ref{min-degree 1}, there must exist a monochromatic $S_c(n,2)$.
\end{proof}

\begin{lemma}\label{min-degree 3}
Let $C$ be a blue cycle of length $2p+2$. If $|S_C| = p+1$ and $|T_C^{2p+1}| = p+1$, then if there exists a vertex $c \in C$ such that $\text{blue-}d_{Z_C}(c) \geq n+1$, then this $2$-coloring admits a monochromatic $S_c(n,2)$.
\end{lemma}

\begin{proof}
Suppose there existed a vertex $c_1 \in C$, with $\text{blue-}d_{Z_C}(q_1) \geq n+1$. Let $C' = C \backslash T_C^{2p+1}$. If any $c \in C'$ satisfied $\text{red-}d_{Z_C}(c) \geq p+1$, then by Lemma \ref{min-degree 1}, this $2$-coloring would admit a monochromatic $S_c(n,2)$. Therefore, we may assume $\forall c \in C'$ that $\text{red-}d_{Z_C}(c) \leq p$ and so $\text{blue-}d_{Z_C}(q) \geq n+1$.

Define $$X = \{ c_i \in C : \exists c_j \in C' \text{ such that } c_i \text{ and } c_j \text{ are endpoints of a path of length } 2p+1 \text{ on } C \}.$$ Conceptually, $X$ is the set of all the vertices on $C$ that are found by following $C$ for $2p+1$ vertices, starting from a vertex in $C'$. Note that $|X| \geq |C'| = p+1$. Suppose some $x \in X$ satisfied $\text{blue-}d_{Z_C}(x) \geq 1$, and let $c \in C'$ be the vertex found by following $C$ for $2p+1$ vertices from $x$. Then $c$ and $x$ would define a blue $S_c(n,2)$, with $n$-star at $c$, and $2$-star at $x$ (where we have used one vertex in $Z_C$, and the remaining vertex on $C$ between $x$ and $c$ to form the $2$-star). Therefore, we may assume that there is a red complete bipartite graph between $X$ and $Z_C$.

Let $X' \subseteq X$ be some subset of $p+1$ vertices of $X$, and set $Y = C \backslash X'$. Note that $|X'| = |Y| = p+1$. Since $X' \subset X$, we may assume by the argument above that there is a red complete bipartite graph between $X'$ and $Z_C$. Now suppose that some $x \in X'$ was connected to some $y \in Y$ by a red edge. Then, using the complete bipartite graph between $X'$ and $Z_C$, we may form a red $S_c(n,1)$, with the $n$-star centered at some vertex in $X' \backslash \{ x \}$, and the $1$-star centered at $x$. Using $y$, this $1$-star can be made into a $2$-star, completing the red $S_c(n,2)$. Therefore, we may assume there is a blue complete bipartite graph between $X'$ and $Y$.

Suppose that there is a vertex $y \in Y$ such that $\text{blue-}d_{Z_C}(y) \geq n+1$. Now if any vertex $y' \in Y \backslash \{ y \}$ has $\text{blue-}d_{Z_C}(y') \geq 1$, then it will complete a blue $S_c(n,2)$. This blue $S_c(n,2)$ can be formed using the blue complete bipartite graph between $X'$ and $Y$, to build a blue path of length $2p+1$ with one endpoint at $y$ and the other endpoint at $y'$; this path defines a blue $S_c(n,2)$ with $n$-star at $y$, and $2$-star at $y'$ (where we have used the remaining unused vertex in $X'$ to complete the $2$-star, if necessary). Therefore, we may assume there is a red complete bipartite graph between $Y \backslash \{ y \}$ and $Z_C$. Now this implies that $\forall z \in Z_C$, 
\begin{align*}
    \text{red-}d_Y(z) &\geq p\\
    &\geq n+2.
\end{align*}
Now using the red complete bipartite graph between $X'$ and $Z_C$, there must exist a red path of length $2p+1$ with both endpoints in $Z_C$. By the above reasoning, both endpoints have sufficient red degree to complete the red $S_c(n,2)$.

Therefore we may assume that $\forall y \in Y$, $\text{red-}d_{Z_C}(y) \geq p + 1$. Since $$p+1 \leq |Z_C| \leq 2(p+1) - 2,$$ by Lemma \ref{Jackson - 1} there must be a blue cycle of length $2p + 2$ between $Y$ and $Z_C$. This cycle contains a red path of length $2p + 1$ with both endpoints in $Z_C$. Since there is a red complete bipartite graph between $Z_C$ and $X'$, the endpoints have sufficient red degree to complete a red $S_c(n,2)$. 

\end{proof}

\begin{lemma} \label{S>p}
If $|S| \geq p+1$, and $n \leq p-2$, then every such $2$-coloring of a $K_r$ admits a monochromatic $S_c(n,2)$, for $r \geq n + 3p + 3$. 
\end{lemma}

\begin{proof}
Clearly it is enough to prove the lemma for $r = n + 3p + 3$. If $|S| \geq p+1$, then by the results of Corollary \ref{min-degree 2} and Lemma \ref{min-degree 3}, we may assume that $\forall q \in Q$, $\text{blue-}d_Z(q) \leq n$ and so $\text{red-}d_Z(q) \geq p+1$. Since $$\frac{(p+1)(2p+2)}{n + p + 1} > p + 1,$$ there must be a vertex $z' \in Z$ such that $\text{red-}d_Q(z') \geq p + 2$. We shall now consider two cases based on the size of $|S|$.

\begin{itemize}
    \item If $|S| = p+1$, then let $W \subset Q$ be a subset of $R(z') \cap Q$ of size $n+2$, such that there exists a vertex $w_1 \in W$ that is not in $S$. Such a subset must exist since $|R(z') \cap Q| \geq p+2$, but $|S| = p+1$. In particular, let $W = \{ w_1, w_2, \ldots, w_{n+1}, w_{n+2} \}$, where $w_1 \notin S$. Finally, let $w_{n+2}$ be chosen so that $w_{n+2}$ is not adjacent to $w_1$ on $Q$.
    \item If $|S| \geq p+2$, then let $W \subset Q$ be any subset of $R(z') \cap Q$ of size $n+2$, where as above, $W = \{ w_1, w_2, \ldots, w_{n+1}, w_{n+2} \}$. Similarly, choose $w_{n+2}$ so that it is not adjacent to $w_1$ on $Q$.
\end{itemize}

Given the above definitions of $W$, now choose a subset of $p+1$ vertices of $Q$ including $w_{n+2}$, but excluding $w_{1}, \ldots, w_{n+1}$. Furthermore, choose these $p+1$ vertices so that neither $w_1$, nor the two vertices adjacent to $w_1$ on $Q$ are chosen. If $w_1 = q_i$, then removing $q_{i-1}$, $q_i$, $q_{i+1}$, and $W \backslash \{ w_{n+2} \}$ from consideration leaves at least 
\begin{align*}
    2p + 2 - (n + 1) - 2 &= 2p - n - 1\\ 
    &\geq p + 1
\end{align*} vertices to choose from. Call this set of $p+1$ chosen vertices $A$. Since $A \subset Q$, it follows that $\forall a \in A$, $\text{red-}d_Z(a) \geq p+1$. Since $$p+1 \leq |Z| \leq 2(p+1) - 2,$$ by Lemma \ref{Jackson - 1} there must be a red cycle of length $2p+2$ between $A$ and $Z$, call this red cycle $Q'$.

Now, either $z' \in Q'$ or $z' \notin Q'$. First, suppose that $z' \in Q'$. Then following $Q'$ for $2p+1$ vertices from $z'$ ends at some other $z_1 \in Z$, where $z_1 \neq u$. Note that such a path always exists since $Q'$ can be followed in two directions; if following $Q'$ for $2p+1$ vertices from $z'$ in the ``forward" direction ends at $u$, then simply follow $Q'$ in the ``backward" direction instead. Since $Q'$ contains $w_{n+2}$, and avoids $w_1, \ldots w_{n+1}$, this path defines a red $S_c(n+1, 1)$, with the $1$-star at $z_1 \in Z \backslash \{ u \}$, and the $n+1$-star centered at $z'$. Alternatively suppose that $z' \notin Q'$. Then following $Q'$ for $2p$ vertices from $w_{n+2}$ will end at some $z_1 \in Z$, where again the direction is chosen so that $z_1 \neq u$. Appending $z'$ to $w_{n+2}$ will again result in a red path of length $2p+1$ between $A$ and $Z$, with one endpoint at $z'$ and the other endpoint at $z_1$. As before, since $Q'$ contains $w_{n+2}$, and avoids $w_1, \ldots w_{n+1}$, this path defines a red $S_c(n+1, 1)$. In either case, there is the desired $S_{c}(n+1, 1)$, with $n$-star at $z'$, and $1$-star at $z_1$. Now observe that this red $S_{2p + 1}(n+1, 1)$ occupies $p + 1$ vertices in $Z$. If $z_1$ is connected to any of the remaining $n$ vertices in $Z$ by a red edge, it will complete a red $S_c(n,2)$. Furthermore, if $z_1$ is connected to any of the vertices in $Q \backslash A$ by a red edge, it will also complete a red $S_{c}(n,2)$. Therefore, $z_1$ must be connected to every vertex in $Q \backslash A$ by a blue edge. Since $\{ q_{i-1}, q_i, q_{i+1} \} \subset A$, this implies that $z_1$ must be connected to every $\{ q_{i-1}, q_i, q_{i+1} \}$ by blue edges (where recall $q_i = w_1$).

Let $$Q'' = \{ q_1, q_2, \ldots, q_{i-1}, z, q_{i+1}, \ldots, q_{2p+1}, q_{2p+2} \}.$$ If $|S| \geq p + 2$, then this is a blue cycle of length $2p + 2$ with $|S_{Q''}| \geq p+1$ (since $z \neq u$, and removing $q_i$ removes at most one vertex that belongs to $S$). Also, $z$  has blue degree $n + 1$ off the cycle, since it is connected to $n$ unused vertices in $Z$ by a blue edge, and also $q_i$ which is no longer on the cycle. Now by either Corollary \ref{min-degree 2} or Lemma \ref{min-degree 3} there must be a monochromatic $S_c(n,2)$. Alternatively if $|S| = p+1$, then in this case recall that $q_i$ was chosen specifically such that $q_i \notin S$, and so replacing $q_i$ with $z$ does not remove any vertices in $S$ from Q''. Now by Lemma \ref{min-degree 3} there must be a monochromatic $S_c(n,2)$.
\end{proof}

Combining the results of Lemma \ref{S<p}, Lemma \ref{S=p}, and Lemma \ref{S>p}, the proof of Theorem \ref{upper bound 2} is complete.

\section{Unsolved Problems and Further Results} \label{Unsolved Problems}

Here we have proved the ramsey numbers of odd-linked double stars in all cases where $n \geq c$, and $n \leq \lfloor \frac{c}{2} \rfloor - 2$, $m = 2$. However, our work leaves open the result for all other cases with $n < c$. We possess partial results in some other cases with $n < c$, but most cases remain open.

Beyond the remaining cases of odd-linked double stars, this result more generally prompts the question of the ramsey numbers of the even-linked double stars. Surprisingly, the ramsey numbers of even-linked double stars have proved to be significantly more challenging to determine than odd-linked double stars. There are, however, some results for specific cases of even-linked double stars:

\begin{itemize}
    \item First we simply note that if $c = 2p$ the ramsey numbers of $S_{c}(n,0)$ or $S_c(n,1)$ are solved exactly by the result of Lemma \ref{Yu_Li}.
    \item The ramsey numbers of $S_4(n,m)$ (example in Figure 2 below) were solved by Burr and Erd\H{o}s to be $$r(S_4(n,m)) = \text{max}(2n + 3, n + 2m + 5).$$ See in \cite{Burr and Erdos} for detailed proof.
    
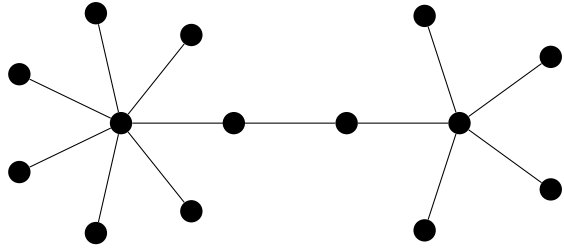
\begin{figure}[ht]
\centering
\begin{center}
\begin{tikzpicture}[transform shape,scale = 0.5, every node/.style={scale=0.7}]
    \node[circle,fill,inner sep=0.3cm] (center) at (0,0) {};
\foreach \phi in {1,...,6}{
    \node[circle,fill,inner sep=0.3cm]  (v_\phi) at (360/7  * \phi:3cm) {};
         \draw[black] (v_\phi) -- (center);
      }
    \node[circle,fill,inner sep=0.3cm] (center3) at (3,0) {};
    \node[circle,fill,inner sep=0.3cm] (center4) at (9,0) {};
    \foreach \phi in {1,...,5}{
    \node[shift = {(90/7,0)}, circle,fill,inner sep=0.3cm]  (v1_\phi) at (36 + 360/5  * \phi:3cm) {};
         \draw[black] (v1_\phi) -- (center4);
      }
    \draw[black] (center) -- (v1_2);
   \end{tikzpicture} 
\end{center}
\caption{Example of a $S_4(6,4)$}
\end{figure}

    \item Finally, the ramsey numbers of $S_2(n,m)$ (example in Figure 3 below) were solved, in most cases, by Grossman et. Al. In particular, they show that for all $n \leq \sqrt{2}m$ and $n \geq 3m$, the ramsey numbers of $S_2(n,m)$ satisfy

\[
    r(S_2(n,m)) = \left\{\begin{array}{lr}
        \text{max}(n + 2m + 1, 2n + 2), & \text{for } n \text{ odd}, m \leq 2\\
        \text{max}(n + 2m + 2, 2n + 2), & \text{otherwise}
        \end{array}\right\}.
  \]
  
  See \cite{GHK} for detailed proof. It was conjectured in \cite{GHK} that the above result holds in the range $\sqrt{2}m < n < 3m$ as well. Surprisingly however, Norin et. Al showed in \cite{Norin} that the conjecture is false, as well as providing asymptotic results for the ramsey numbers of $S_2(n,m)$ in the gap. The exact value of $r(S_2(n,m))$ in the range $\sqrt{2}m < n < 3m$, however, remains an open question.
 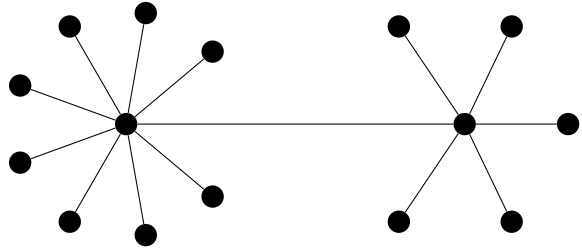
\begin{figure}[ht]
\centering
\begin{center}
\begin{tikzpicture}[transform shape,scale = 0.5, every node/.style={scale=0.7}]
    \node[circle,fill,inner sep=0.3cm] (center) at (0,0) {};
\foreach \phi in {1,...,8}{
    \node[circle,fill,inner sep=0.3cm]  (v_\phi) at (360/9  * \phi:3cm) {};
         \draw[black] (v_\phi) -- (center);
      }
    \node[circle,fill,inner sep=0.3cm] (center4) at (9,0) {};
    \foreach \phi in {1,...,5}{
    \node[shift = {(12.5,0)}, circle,fill,inner sep=0.3cm]  (v1_\phi) at (180 + 360/6  * \phi:3cm) {};
         \draw[black] (v1_\phi) -- (center4);
      }
    \draw[black] (center) -- (center4);
   \end{tikzpicture} 
\end{center}
\caption{Example of a $S_2(8,5)$}
\end{figure}
  
\end{itemize}

We generally ask then, for either $c \geq 6$ or $c = 2$, what are the ramsey numbers of the \textit{even-linked double stars}?

\bibliographystyle{unsrt}  


\begin{thebibliography}{1}

\bibitem{Radziszowski}
Radziszowski, S. P. 1994. Small RAMSEY numbers. Electron. J. Combinatorics 1. Revision 16: January 15, 2021.]]

\bibitem{GHK}
J. W. Grossman, F. Harary, and M. Klawe. Generalized Ramsey theory for graphs. X. Double stars. Discrete Mathematics, 28(3):247–254, 1979.

\bibitem{Norin}
Norin, S., Sun, Y. R., and Zhao, Y. Asymptotics of Ramsey numbers of
double stars. Preprint 2016, arXiv:1605.03612.

\bibitem{Yu and Li}
Pei Yu and Yusheng Li,  All  Ramsey  Numbers  for  Brooms  in Graphs, \textit{Electronic  Journal  of  Combinatorics}, http://www.combinatorics.org, 2016.

\bibitem{Jackson}
B. Jackson, Cycles in bipartite graphs, J. Combin. Theory Ser. B, 30 (1981), 332-342. 

\bibitem{FR_cycles}
R. Faudree and R. Schelp. All Ramsey numbers for cycles in graphs. \textit{Discrete Math.}, 8:35-52,1974

\bibitem{Rosta}
V. Rosta, On a Ramsey Type Problem of J.A. Bondy and P. Erd\H{o}s, I $\&$ II,Journal of CombinatorialTheory,Series B,15(1973) 94-120.

\bibitem{Parsons}
T.D. Parsons. Path-star Ramsey numbers. J. Combin. Theory Ser. B, 17 (1974), pp. 51-58

\bibitem{Burr and Erdos}
S. Burr and P. Erd\H{o}s, Extremal Ramsey theory for graphs, Utilitas Mathematica 9 (1976) 246-258.


\end{thebibliography}

\end{document}